\documentclass{article}
\usepackage{amsmath,amssymb,amsfonts,amsthm,latexsym,amstext,amscd}
\usepackage{epsfig,graphics,graphicx,color}
\usepackage{enumerate}
\usepackage{hyperref}

\newcommand{\bsx}{\boldsymbol{x}}%
\newcommand{\bsy}{\boldsymbol{y}}%
\newcommand{\bsz}{\boldsymbol{z}}%
\newcommand{\bsc}{\boldsymbol{c}}%
\newcommand{\bsT}{\boldsymbol{T}}%
\newcommand{\bszero}{\boldsymbol{0}}%

\newcommand{\NN}{\mathbb N}
\newcommand{\ZZ}{\mathbb Z}

\newcommand{\QQ}{\mathbb Q}
\newcommand{\PP}{\mathbb P}

\newcommand{\FF}{\mathbb{F}}

\newtheorem{theorem}{Theorem}
\newtheorem{lemma}{Lemma}

\newtheorem{example}{Example}
\newtheorem{prop}{Proposition}

\newtheorem{defi}{Definition}

\begin{document}

\title{Kronecker-Halton sequences in $\FF_p((X^{-1}))$}
\author       {Roswitha Hofer\thanks{supported by the Austrian Science Fund (FWF):
Project F5505-N26, which is a part of the Special Research Program
``Quasi-Monte Carlo Methods: Theory and Applications''}}

\maketitle

\begin{abstract}
In this paper we investigate the distribution properties of hybrid sequences which are made by combining Halton sequences in the ring of polynomials and digital Kronecker sequences. We give a full criterion for the uniform distribution and prove results on the discrepancy of such hybrid sequences. 

\end{abstract}
\vskip0.5cm%
\maketitle
\vskip0.5cm%
\noindent{Keywords: hybrid sequences, digital Kronecker sequences, Halton-type sequences, discrepancy}

\section{Preliminaries}

Let $(\bsz_n)_{n\geq 0}$ be a sequence in the $s$-dimensional unit cube $[0,1)^s$, then the \emph{discrepancy} $D_N$ of the first $N$ points of the sequence is defined by 
$$D_N=\sup_{B\subseteq [0,1)^s}\left|\frac{A_N(B)}{N}-\lambda(B)\right|$$
where $$A_N(B):=\#\{n:0\leq n<N, \bsz_n\in B\},$$
$\lambda$ is the $s$-dimensional Lebesgue measure and the supremum is taken over all axis-parallel subintervals $B\subseteq [0,1)^s$. 
When restricting the supremum over all axis-parallel subintervals with the lower left point in the origin, then we obtain the \emph{star discrepancy} $D_N^*$ of the first $N$ points of the sequence. It is easy to see that 
$D_N^*\leq D_N\leq 2^sD_N^*$. 
The sequence $(\bsz_n)_{n\geq 0}$ is called \emph{uniformly distributed} if $\lim_{N\to\infty}D_N=0$. 

It is frequently conjectured in the theory of irregularities of distribution, that for every sequence $(\bsz_n)_{n\geq 0}$ in $[0,1)^s$ we have 
$$D_N\geq c_s\frac{\log^sN}{N}$$
for a constant $c_s>0$ and for infinitely many $N$. In the following we will abbreviate this to $D_N\gg_s \frac{\log^sN}{N}$. Therefore sequences whose discrepancy satisfies $D_N\leq C_s \log^sN/N$ for all $N$ with a constant $C_s>0$ that is independent of $N$ (or $D_N\ll_s \log^sN/N$), are called \emph{low-discrepancy sequences}. 

Well-known examples of low-discrepancy sequences are the $s$-dimensional Halton sequences, digital $(t,s)$-sequences, and one-dimensional Kronecker sequences $(\{n\alpha\})_{n\geq 0}$ with $\alpha$ irrational and having bounded continued fraction coefficients. For the sake of completeness we define the Halton sequences, the Kronecker sequences, and the digital $(t,s)$-sequences. 

For the Halton sequence \cite{halton} $(\bsy_n)_{n\geq 0}$ we choose $s$ different pairwise coprime bases $b_1,\ldots,b_s\geq 2$ and construct the $i$th component $y_n^{(i)}$ of the $n$th point $\bsy_n=(y_n^{(1)},\ldots,y_n^{(s)})$ by representing $n=n_0^{(i)}+n_1^{(i)}b_i+n_2^{(i)}b_i^2+\cdots$ in base $b_i$ and set $$y_n^{(i)}=n_0^{(i)}/b_i+n_1^{(i)}/b_i^2+n_2^{(i)}/b_i^3+\cdots.$$

The $s$-dimensional Kronecker sequence related to the real numbers $\alpha_1,\ldots,\alpha_s$ is defined by  $\left(\bsx_n=(\{n\alpha_1\},\,\dots,\,\{n\alpha_s\})\right)_{n\geq 0}$, where $\{\cdot\}$ denotes the fractional part operation.  It is well-known to be uniformly distributed if and only if $1,\,\alpha_1,\,\dots,\,\alpha_s$ are linearly independent over $\mathbb{Q}$. 

For the digital $(t,s)$-sequences in the sense of Niederreiter \cite{nie92} we start with the more general, digital $(\bsT,s)$-sequences in the sense of Larcher and Niederreiter, see \cite{ln95}. 

\begin{defi}
Choose $s$, $\NN\times\NN_0$-matrices $C^{(1)},\,\dots,\,C^{(s)}$ over $\FF_p$, $p$ prime. To generate the $i$th coordinate $x_n^{(i)}$ of $\bsx_n$, represent the integer $n$ in base $p$
$$ n = n_0+n_1p+\dots+n_rp^r, $$
set
$$ \vec n :=(n_0,\,\dots,\,n_r,\,0,\,0,\,\dots)^T $$
and 
$$ C^{(i)}\cdot\vec n =:(y_1^{(i)},\,y_2^{(i)},\,\dots)^T. $$
Further
$$ x_n^{(i)}:=\frac {y_1^{(i)}}p+\frac {y_2^{(i)}}{p^2}+\dots\,. $$
Then $(\bsx_n)_{n\ge 0}$ is called a\emph{ digital $(\bsT,s)$-sequence over $\FF_p$}, where
the parameter $\bsT$ is defined as follows. 
For every $m\in\NN$ let $\bsT(m)$, satisfying $0\le \bsT(m)\le m$,
be such that for all $d_1,d_2,\ldots,d_s\in\NN_0$ with $d_1+d_2+\cdots
+
d_s =m-\bsT(m)$ the $(m-\bsT(m))\times m$-matrix consisting of the\\

left upper $d_1 \times m$-submatrix of $C_1$ together with the

left upper $d_2 \times m$-submatrix of $C_2$ together with the

\ \vdots

left upper $d_s \times m$-submatrix of $C_s$\\

\noindent has rank $m-\bsT(m)$.
 If $\bsT(m) \le t$ for all $m$, then we speak
of a digital $(t,s)$-sequence over $\FF_p$.
\end{defi}
Note that in the definition above as well as in the following we do not distinguish between the elements of $\FF_p$ and the elements in the set $\{0,1,\ldots,p-1\}$. 
It is well known that a digital $(\bsT,s)$-sequence is uniformly distributed if $\lim_{m\to\infty}(m-\bsT(m))=\infty$. A necessary and sufficient condition for the uniform distribution is that the rows of the generating matrices $C^{(1)},\dots,C^{(s)}$ altogether are linearly independent over $\FF_p$, i.e., that any finite set of rows of $C^{(1)},\dots,C^{(s)}$ is linearly independent over $\FF_p$. 

There are many known examples of digital $(t,s)$-sequences, see for instance \cite{faure,Hofer,HoNie,nie,NieXing,NieYeo,sobol}. 
For the sake of completeness we give the definition of $(t,m,s)$-nets and $(t,s)$-sequences which was introduced by Niederreiter \cite{N87}. 

\begin{defi}For a given dimension $s$, an integer base $p\geq2$, a positive integer $m$ and an integer $t$ with $0\leq t \leq m$, a finite sequence of $p^m$ points in $[0,1)^s$ is called a \emph{$(t,m,s)$-net in base $p$} if each subinterval of the form $I=\prod_{i=1}^s[a_i/p^{d_i},(a_i+1)/p^{d_i})$, where $a_i,\,d_i$ are nonnegative integers satisfying $a_i<p^{d_i}$ for all $1\leq i\leq s$ and $d_1+\cdots + d_s=m-t$, contains exactly $p^t$ points. An infinite sequence $(\bsx_n)_{n\geq 0}\in[0,1)^s$ is called a \emph{$(t,s)$-sequence in base $p$} if for all integers $m,\,k$, satisfying $m> t$ and $k\geq 0$, the point set consisting of $\bsx_{kp^m},\bsx_{kp^m+1},\ldots, \bsx_{(k+1)p^m-1}$ forms a $(t,m,s)$-net in base $p$. 
\end{defi}

The star discrepancy of a $(t,m,1)$-net in base $p$ satisfies $ND_N^*\leq p^t$ (see e.g. \cite[Theorems~4.5 and 4.6]{niesiam}). 
For more information on $(t,s)$-sequences we refer the interested reader to \cite{DP10,niesiam}.

In the following we write $\FF_p[X]$ for the ring of polynomials over $\FF_p$, $\FF_p(X)$ for the field of rational functions over $\FF_p$, and $\FF_p((X^{-1}))$ for the field of formal Laurent series over $\FF_p$. 

Let $L\in\FF_p((X^{-1}))$ then there exist two expansions of $L$. 
The first is its Laurent series 
$$L=\sum_{i=w}^\infty a_iX^{-i}$$
with $a_i\in\FF_p$, $w\in\ZZ$ such that $a_w\neq 0$. Then $\nu(L)=-w$. 
We define the fractional part $\{L\}$ of $L$ by 
$$\{L\}=\sum_{i=\max(1,w)}^\infty a_iX^{-i}$$
The second is the continued fraction expansion $$L=[A_0;A_1,A_2,\ldots]$$
with $A_i\in\FF_{p}[X]$ for $i\geq 0$ and $\deg(A_i)\geq 1$ for $i\geq 1$. The expansion is finite for rational $L$ and infinite else. For $h\geq 0$ the $h$th convergent $P_h/Q_h$ of $L$ is defined by $P_h/Q_h=[A_0;A_1,\ldots,A_h]$, where $P_h,Q_h\in\FF_p[X]$ and $\gcd(P_h,Q_h)=1$. The degree of $Q_h$ is often abbreviated to $d_h$ and satisfies $d_h=\sum_{i=1}^h\deg(A_i)$. Furthermore, $\nu(L-P_h/Q_h)=-d_h-d_{h+1}$ for $k\geq 0$ and for all $k\in\FF_p[X]$ with $0\leq \deg(k)<d_{h+1}$ we have $$\nu(L-b/k)\geq \nu(L-P_h/Q_h) \mbox{ for all $b\in\FF_p[X]$.}$$ (For informations on continued fractions and convergents we refer the interested reader to the Appendix B of \cite{niesiam}.)

We define the \emph{one-dimensional digital Kronecker sequence} $(x_n)_{n\geq 0}$ using $L$. Write $n$ in base $p$, $n=n_0+n_1p+\cdots+n_rp^r$ associate the polynomial $n(X)\in\FF_p[X]$ as $n(X)=n_0+n_1X+\cdots+n_rX^r$. For the $n$th point compute 
$\{n(X)L(X)\}$ and evaluate it by setting $X$ equal $p$. This sequence, often abbreviated to $(\{n(X)L(X)\})_{n\geq 0}$, can be interpreted as digital sequence with generating matrix $C$ given by
$$C=\begin{pmatrix} c_{1,0}&c_{1,1}&c_{1,2}&\ldots \\
c_{2,0}&c_{2,1}&c_{2,2}&\ldots  \\
c_{3,0}&c_{3,1}&c_{3,2}&\ldots \\
\vdots &\vdots&\vdots&\ddots\end{pmatrix}=\begin{pmatrix} a_1&a_2&a_3&\ldots \\
a_2&a_3&a_4&\ldots \\
a_3&a_4&a_5&\ldots \\
\vdots &\ddots&\ddots&\ddots\end{pmatrix}.$$
Straightforward we define the $s$-dimensional Kronecker sequence determined by $L_1,\ldots,L_s$ by just juxtaposing the one-dimensional Kronecker sequences using $L_i$, i.e., $(\{n(X)L_1(X)\},\ldots,\{n(X)L_s(X)\})_{n\geq 0}$. 

Let $\mathcal{H}$ be the set of formal Laurent series $L$ in $\FF_p((X^{-1}))$ with $\nu(L)<0$. Let $h$ be the normalized Haar-measure on $\mathcal{H}$.

There are many analogies between the ordinary Kronecker sequence and the digital Kronecker sequence. 


\begin{itemize}
\item The $s$-dimensional Kronecker sequence associated with $\alpha_1,\ldots,\alpha_s$ is uniformly distributed if and only if $1,\alpha_1,\ldots,\alpha_s$ are linearly independent over $\QQ$. 
\item Let $L_1,\ldots,L_s\in\FF_p((X^{-1}))$. 
The $s$-dimensional Kronecker sequence determined by $L_1,\ldots,L_s$ is uniformly distributed if and only if $1,L_1,\ldots,L_s$ are linearly independent over $\FF_p(X)$ \cite[Theorem~1]{LarNie2}. 

\item The one-dimensional Kronecker sequence associated with $\alpha$ is a low-discrepancy sequence if the continued fraction coefficients of $\alpha$ are bounded. 
\item The one-dimensional digital Kronecker sequence determined by $L$ is a low-discrepancy sequence if the continued fraction coefficients of $L$ have bounded degrees \cite[Theorem~4.48]{niesiam}.
\item The one-dimensional Kronecker sequence associated with $\alpha\in[0,1]$ satisfies for all $\epsilon>0$, $ND_N\ll_{\alpha,\epsilon} \log^{1+\epsilon} N$
for allmost all $\alpha\in[0,1]$ in the sense of Lebesgue measure. 
\item The one-dimensional digital Kronecker sequence determined by $L\in\mathcal{H}$ satisfies for every $\epsilon>0$, $ND_N\ll_{L,p,\epsilon} \log^{1+\epsilon} N$ for almost all $L\in\mathcal{H}$ in the sense of Haar-measure \cite[Corollary~1]{LarNie2}. (See \cite{Larcher2} for the multidimensional version.)
\end{itemize}

There exists also an analog to the Halton sequence in $\FF_p[X]$. 

Let $\FF_p$ be a finite prime field and $b(X)$ be a nonconstant monic polynomial over $\FF_p$ of degree $e$. We define the van der Corput sequence in base $b(X)$ as introduced in \cite{HoferRB,NieYeo}. For the $n$th point $y_n$ regard the base $p$ representation of $n=n_0+n_1p+n_2p^2+\cdots$ and associate the polynomial $n(X)=n_0+n_1X+n_2X^2+\cdots $. Compute the base $b(X)$ representation of $n(X)$, 
$$n(X)=a_0(X)+a_1(X)b(X)+a_2(X)b^2(X)+\cdots$$
with $\deg (a_i(X))<e$, and set $$y_n=\sum_{i=0}^\infty \frac{a_i(p)}{p^{e(i+1)}}.$$

Straightforward we define the $s$-dimensional Halton sequences in bases $(b_1(X),\ldots,b_s(X))$ by just juxtaposing the van der Corput sequences in bases $b_i(X)$. 
\begin{example}
Let $p=2$ and $b(X)=X$. Then the van der Corput sequence in base $X$ over $\FF_2$ is the ordinary van der Corput sequence in base $2$. 
Let $b_1(X)=X$ and $b_2(X)=X+1$ then the Halton sequence in bases $(X,X+1)$ over $\FF_2$ corresponds with the two-dimensional Sobol sequence \cite{sobol} and with the Faure sequence in base $2$ \cite{faure}. 
\end{example}

Again there exist many analogies between the two types of Halton sequences.
\begin{itemize}
\item The Halton sequence in bases $b_1,\ldots,b_s\geq 2$ is uniformly distributed if and only if the bases are pairwise coprime. 
\item The Halton sequence in monic nonconstant bases $b_1(X),\ldots,b_s(X)$ is uniformly distributed if and only if the bases are pairwise coprime. 
\item The Halton sequence in pairwise coprime bases $b_1,\ldots,b_s\geq 2$ is a low-discrepancy sequence. 
\item The Halton sequence in monic nonconstant pairwise coprime bases $b_1(X),\ldots,b_s(X)$ is a low-discrepancy sequence. Indeed it is a $(t,s)$-sequence in base $p$ where $t=\sum_{i=1}^s(e_i-1)$ with $e_i=\deg(b_i(X))$ (see e.g. \cite{HoferRB}). 
\item Let $I$ be an elementary interval of the form 
$$I=\prod_{i=1}^s\left[\frac{a_i}{b_i^{d_i}},\frac{a_i+1}{b_i^{d_i}}\right)$$
with $d_i\geq 0$ and $0\leq a_i<b_i^{d_i}$. Then a point $\bsx_n$ of the Halton sequence in pairwise coprime bases $b_1,\ldots,b_s\geq 2$ is contained in $I$ if and only if 
$$n\equiv R\pmod{\prod_{i=1}^sb_i^{d_i}}$$
where $R$ is determined by the $a_i$. 
\item Regard the Halton sequence in monic nonconstant pairwise coprime bases $b_1(X),\ldots,b_s(X)$ with $\deg(b_i(X))=e_i$. 
Let $I$ be an elementary interval of the form 
$$I=\prod_{i=1}^s\left[\frac{a_i}{p^{e_id_i}},\frac{a_i+1}{p^{e_id_i}}\right)$$
with $d_i\geq 0$ and $0\leq a_i<p^{e_id_i}$. Then a point $\bsx_n$ of the Halton sequence is contained in $I$ if and only if 
$$n(X)\equiv R(X)\pmod{\prod_{i=1}^sb_i(X)^{d_i}}$$
where $R(X)$ is determined by the $a_i$. 
\end{itemize}

\section{Kronecker-Halton sequences}

In the last decade hybrid sequences were actively studied (see for instance \cite{GHN13,HellekalekKritzer,Hofer09a,HoferKritzer,hklp,hoflar,HoferLarcher,HoferPuchhammer,Kritzer,KriPil12,Niederreiter09,Niederreiter10c,Niederreiter11b,NiederreiterWinterhof11}). The idea of building hybrid sequences is to concatenate the components of two or more different types of low-discrepancy sequences or in the original idea of Spanier \cite{Spanier} to combine deterministic sequences with pseudo-random sequences. The intentions are multiple; combining the different structures and/ or advantages of the component sequences, providing new types of sequences; discovering new types of low-discrepancy sequences. The difficulty we face when studying the distribution of hybrid sequences is to work out proper methods which can handle the different structures of the component sequences. Hybrid sequences with one or more digital component sequences turned out to be particularly hard-to-study objects. There are a few results in \cite{Hofer09a,HoferKritzer,hoflar,HoferPuchhammer,KriPil12}.  

The $(s+t)$-dimensional hybrid sequences made of Kronecker sequences related to $\alpha_1,\ldots,\alpha_s$ and Halton sequences in bases $b_1,\ldots,b_t$ are well studied objects \cite{HoferLarcher,Larcher,DrmotaHoferLarcher}. It is known that ... 
\begin{enumerate}
\item ... the hybrid sequence is uniformly distributed if and only if the component sequences are uniformly distributed. (This follows from a more general result in \cite{HoferKritzer}.) 
\item ... the hybrid sequence satisfies for almost all $(\alpha_1,\ldots,\alpha_s)$ in the sense of Lebesgue measure a discrepancy bound in the style of $$ND_N\ll_{\epsilon,\alpha_1,\ldots,\alpha_s,b_1,\ldots,b_s}\log^{s+t+\epsilon} N$$
(see \cite{Larcher,HoferLarcher}).
\item ... the uniformly distributed hybrid sequence in the case where $s=1$ satisfies a discrepancy bound of the form $$ND_N\ll_{\alpha_1,b_1,\ldots,b_t} N^{1/2}\log^t N$$
if $\alpha_1$ has bounded continued fraction coefficients and this bound is best possible up to the $\log$ term \cite[Theorem~2]{DrmotaHoferLarcher}.  
\item ... the uniformly distributed hybrid sequence in the case where $s=1$ satisfies a discrepancy bound of the form $$ND_N\ll_{\epsilon,\alpha_1,b_1,\ldots,b_t} N^{\epsilon}$$
if $\alpha_1$ is an irrational algebraic \cite[Theorem~1]{DrmotaHoferLarcher}.
\end{enumerate}

In this paper we built hybrid sequences whose component sequences stem from the analogs of Kronecker sequences and Halton sequences and provide new results on hybrid sequences built of digital component sequences. We prove an analog of item 1 in Theorem~\ref{thm:1}.

\begin{theorem}\label{thm:1}
Let $p\in\PP$ and $s,t\in\NN$. Let $b_1(X),\ldots,b_t(X)\in\FF_p[X]$ be monic pairwise coprime nonconstant polynomials. Furthermore, let $L_1,\ldots,L_s\in\FF_p((X^{-1}))$ be such that they are together with $1$ linearly independent over $\FF_p(X)$. Let $(\bsx_n)_{n\geq 0}$ be the digital Kronecker sequence related to $L_1,\ldots,L_s$ and $(\bsy_n)_{n\geq 0}$ the digital Halton sequence in bases $b_1(X),\ldots,b_t(X)$. 
Then the hybrid sequence $(\bsx_n,\bsy_n)_{n\geq 0}$ is uniformly distributed. 
\end{theorem}

Furthermore, we prove an analog to item 3 in Theorem~\ref{thm:2} which by Theorem~\ref{thm:3} is best possible up to the $\log $ terms. Theorem~\ref{thm:3} already indicates that an analog of item 4 does not hold true as the $L$ considered there is an algebraic one. For the proof of item 4 in \cite{DrmotaHoferLarcher} an essential tool is Ridout's $p$-adic version of the Thue-Siegel-Roth Theorem \cite{ridout}. Now it turns out \cite{BugeaudMathan} that the $p$-adic version of the Thue-Siegel-Roth Theorem is far from being true in the field of power series in positive characteristics. 

\begin{theorem}\label{thm:2}
Let $p\in\PP$, $t\in\NN$, $b_1(X),\ldots,b_t(X)$ be monic pairwise coprime nonconstant polynomials over $\FF_p$. Let $(\bsy_n)_{n\geq 0}$ be the Halton sequence in bases $b_1(X),\ldots,b_t(X)$. 
Let $L$ having continued fraction coefficients of bounded degrees. Let $(x_n)_{n\geq 0}$ be the Kronecker sequence associated with $L$. Then the discrepancy of the Kronecker-Halton sequence $(x_n,\bsy_n)_{n\geq 0}$ satisfies 
$$ND_N\ll_{p,L,b_1,\ldots,b_t,t} N^{1/2}\log^{t+1} N.$$ 
\end{theorem}

\begin{theorem}\label{thm:3}
We regard the two-dimensional sequence $(x_n,y_n)_{n\geq 0}$, where $(x_n)_{n\geq 0}$ is the digital Kronecker sequence associated with $$L=[0;X,X^2,X,X^2,X,X^2,\ldots]\in\FF_2((X^{-1}))$$ and $(y_n)_{n\geq 0}$ is the Halton sequence in base $X$ over $\FF_2$ (i.e., the van der Corput sequence in base $2$). Then 
$$ND_N\gg \sqrt{N}$$
\end{theorem}

Finally, we prove a metrical result in the sense of the second item. 

\begin{theorem}\label{thm:4}
Let $p\in\PP$, $t \in\NN$, $b_1(X),\ldots,b_t(X)$ be monic pairwise coprime nonconstant polynomials over $\FF_p$. Let $(\bsy_n)_{n\geq 0}$ be the Halton sequence in bases $b_1(X),\ldots,b_t(X)$. 
Let $L\in\mathcal{H}$ and let $(x_n)_{n\geq 0}$ be the Kronecker sequence associated with $L$. Then the star-discrepancy of the Kronecker-Halton sequence $(x_n,\bsy_n)_{n\geq 0}$ satisfies for all $\epsilon>0$,
$$ND_N\ll_{p,L,b_1,\ldots,b_t,\epsilon} \log^{t+1+\epsilon} N$$
for almost all $L\in\mathcal{H}$ in the sense of Haar-measure.  
\end{theorem}

We prove our theorems in the rest of the paper. Theorem~\ref{thm:1} is treated in Section~\ref{sec:3},  Theorem~\ref{thm:2} in Section~\ref{sec:4},   Theorem~\ref{thm:3} in Section~\ref{sec:5}, and finally Theorem~\ref{thm:4} in Section~\ref{sec:6}.

\section{Proof of Theorem~\ref{thm:1}}\label{sec:3}

For the proof of Theorem~\ref{thm:1} we need the following lemma. 

\begin{lemma}\label{lem:2}
Let $e\in\NN_0$, $B(X),R(X)\in\FF_p[X]$ with $\deg(R(X))<\deg(B(X))=e$ and $B(X)$ is monic. Furthermore let $u\in\NN$ and $K\in\NN_0$. Let $n=Kp^{u+e}, Kp^{u+e}+1,\ldots,(K+1)p^{u+e}-1$. We regard all associated polynomials $n(X)$ that satisfy $n(X)\equiv R(X)\pmod{B(X)}$. Then they are of the form 
$$n(X)=k(X)B(X)+R(X)$$
with $k(X)$ out of the set 
$$k(X)=r(X)+X^uC(X)$$
with a fixed $C(X)\in\FF_p(X)$ and $r(X)$ ranges over all polynomials of degree $<u$. 
\end{lemma}
\begin{proof}
We write $B(X)=X^e+b_{e-1}X^{e-1}+\cdots+b_1X+b_0$ and $n(X)=K(X)X^{u+e}+M(X)$ where $M(X)$ ranges over all polynomials of degree $<u+e$. 
Then $$K(X)X^{u+e}=K(X)B(X)X^u+K(X)(-b_{e-1}X^{e-1}-\cdots-b_1X-b_0)X^u$$
and 
$$M(X)=u(X)B(X)+v(X)$$
where $u(X)$ ranges over all polynomials of degree $<u$ and $v(X)$ over all polynomials of degree $<e$. 
Hence 
\begin{align*}
n(X)&=K(X)X^{u+e}+M(X)\\
&=(K(X)X^u+u(X))B(X)+v(X)+K(X)X^u(-b_{e-1}X^{e-1}-\cdots-b_1X-b_0).
\end{align*}
Now there is a unique ${v}(X)$, say $\overline{v}(X)$, such that 
$$\overline{v}(X)+K(X)X^u(-b_{e-1}X^{e-1}-\cdots-b_1X-b_0)\equiv R(X)\pmod{B(X)}.$$
Now those $n(X)$ satisfying $n(X)\equiv R(X)\pmod{B(X)}$ are of the form 
\begin{align*}
n(X)
&=\big(K(X)X^u+u(X)\big)B(X)+\overline{v}(X)+K(X)X^u(-b_{e-1}X^{e-1}-\cdots-b_1X-b_0)\\
&=\big(K(X)X^u+u(X)\big)B(X)+R(X)+B(X)\big(X^uC_1(X)+C_2(X)\big)\\
&=\Big(\underbrace{\big(K(X)+C_1(X)\big)}_{=:C(X)}X^u+\underbrace{u(X)+C_2(X)}_{=:r(X)}\Big)B(X)+R(X),
\end{align*}
where $C_1(X),C_2(X)$ are fixed polynomials with $\deg(C_2(X))<u$ and $u(X)$ ranges over all polynomials of degree $<u$. Hence $C(X)$ is fixed and $r(X)$ ranges over all polynomials of degree $<u$ and the proof is complete.
\end{proof}

Let $e_j:=\deg(b_j(X))$ for $j=1,\ldots,t$.  
It is sufficient to prove the uniform distribution on elementary intervals of the following form 
$$I=\prod_{i=1}^s\left[\frac{a_i}{p^{d_i}},\frac{a_i+1}{p^{d_i}}\right)\times\prod_{j=1}^t\left[\frac{c_j}{p^{e_jl_j}},\frac{c_j+1}{p^{e_jl_j}}\right) $$
with $l_j,d_i\geq 0$ and $0\leq a_i< p^{d_i}$, $0\leq c_j<p^{e_jl_j}$, as any arbitrary interval can be approximated arbitrarily precise by unions of such elementary intervals. 

Now by the construction of the digital Halton sequence $$\bsy_n\in\prod_{j=1}^t\left[\frac{c_j}{p^{e_jl_j}},\frac{c_j+1}{p^{e_jl_j}}\right)$$
if and only if 
$$n(X)\equiv R(X)\pmod{\underbrace{\prod_{j=1}^tb_j^{l_j}(X)}_{=:B(X)}}$$
where $R(X)$ is uniquely determined by the $c_j$. 
So we have to consider for $\bsx_n$ the subsequence determined by the polynomials 
$n(X)=k(X)B(X)+R(X)$ where $k(X)\in\FF_p(X)$. 
We have 
$$(\{(k(X)B(X)+R(X))L_1\},\cdots,\{(k(X)B(X)+R(X))L_s\})\in \prod_{i=1}^s\left[\frac{a_i}{p^{d_i}},\frac{a_i+1}{p^{d_i}}\right)$$
if and only if 
$$\bsz_k:=(\{k(X)B(X)L_1\},\cdots,\{k(X)B(X)L_s\})\in \prod_{i=1}^s\left[\frac{r_i}{p^{d_i}},\frac{r_i+1}{p^{d_i}}\right)$$
where the $r_i$ are uniquely determined by the $a_i$ and $R(X)$. Let $D^{(1)},\ldots,D^{(s)}$ be the generating matrices associated with $B(X)L_1,\ldots,B(X)L_1$. Now from the fact that $L_1,\ldots,L_s,1$ are linearly independent over $\FF_p(X)$ we know that $B(X)L_1,\ldots,B(X)L_1,1$ are also linearly independent over $\FF_p(X)$. Therefore,, we know that the first $d_1$ rows of $D^{(1)}$ together with the first $d_2$ rows of $D^{(2)}$ together with ... the first $d_s$ rows of $D^{(s)}$ are linearly independent over $\FF_p$. Hence there is a $u\in\NN$ such that when considering any $p^u$ consecutive points where  $k=Up^u,Up^u+1,\ldots(U+1)p^u-1$ then $\prod_{i=1}^s\left[\frac{r_i}{p^{d_i}},\frac{r_i+1}{p^{d_i}}\right)$ contains a fair portion of points, i.e., $p^{u-d_1-\cdots-d_s}$ many. This together with Lemma~\ref{lem:2} implies when considering $(\bsx_n,\bsy_n)$ with $n$ in the range  $n=Kp^{u+\sum_{j=1}^te_jl_j},Kp^{u+\sum_{j=1}^te_jl_j}+1,\ldots,(K+1)p^{u+\sum_{j=1}^te_jl_j}-1$ that exactly $p^{u-d_1-\cdots-d_s}$ points lie in $I$. This yields the uniform distribution on the elementary interval and hence of the hybrid sequence.

\section{Proof of Theorem~\ref{thm:2}}\label{sec:4}

For preparing the proof of Theorem~\ref{thm:2} we state the following proposition, which is interesting on its own. 

\begin{prop} \label{prop:1}
Let $L$ in $\FF_p((X^{-1}))$ such that the degrees of the coefficients $A_d$ in the continued fraction expansion are bounded. 
We define $K(L)=\sup_{d\geq 1}\deg(A_d)$. Furthermore, let $B\in\FF_p[X]$ with $\deg(B)=e$. Then the digital Kronecker sequence associated with $BL$ is a $(t,1)$-sequence over $\FF_p$ with $t=K(L)+e-1$. 
\end{prop}
\begin{proof}The proposition is already known in the case where $B(X)=1$ (see \cite{niesiam}). Now let $BL=\sum_{i=w_B}^\infty a_iX^{-i}$.

It suffices to prove that for $m>t$ the vectors 
$$\bsc_j=(c_{j,0},c_{j,1},\ldots,c_{j,m-1})\mbox{ for }1\leq j\leq m-t$$
are linearly independent. 
Suppose there are $h_1,\ldots,h_{m-t}\in\FF_p$ such that 
$$\sum_{j=1}^{m-t}h_j \bsc_j=\bszero \in\FF_p^{m}$$
where not all $h_j$ are zero. 
Then $$\sum_{j=1}^{m-t}h_ja_{i+j}=0 \mbox{ for }0\leq i \leq m-1.$$
With $h(X)=\sum_{j=1}^{m-t}h_jX^{j-1}$ we obtain 
\begin{align*}
hBL&= \left(\sum_{j=1}^{m-t}h_jX^{j-1}\right)\left(\sum_{i=w_B}^\infty a_iX^{-i}\right)\\
&=\sum_{j=1}^{m-t}h_j\sum_{i=w_B}^\infty a_iX^{-i+j-1}\\
&=\sum_{j=1}^{m-t}h_j\sum_{i=w_B-j}^\infty a_{i+j}X^{-i-1}.
\end{align*}
And so the coefficient of $X^{-i-1}$ is zero for ${i=0,1,\ldots,m-1}$. Thus for a suitable $q\in\FF_p[X]$ we have $\nu(hBL-q)<-m$. Since $\deg(h(X))\leq m-t-1$ we have 
$$\deg(h)+\nu(hBL-q)<m-t-1-m=-t-1=-K(L)-e.$$
On the other hand use the denominators $Q_h$ of the convergents $P_h/Q_h$ to $L$ and choose $d\in\NN$ such that $\deg(Q_{d-1})\leq \deg(hB)<\deg(Q_d)$, then
\begin{align*}
\deg(h)+\nu(hBL-q)&=2\deg(hB)-\deg(B)+\nu((L-q/(Bh)))\\
&\geq 2\deg(Q_{d-1})+\nu((L-q/(Bh)))-\deg(B)\\
&\geq 2\deg(Q_{d-1})+\nu((L-P_{d-1}/Q_{d-1}))-\deg(B)\\
&=2\deg(Q_{d-1})-\deg(Q_{d-1})-\deg(Q_d)-\deg(B)\\
&=\deg(Q_{d-1})-\deg(Q_d)-\deg(B)\\
&=-\deg(A_d)-\deg(B)\geq -K(L)-e.
\end{align*}
This is a contradiction. 
\end{proof}

\begin{example}\label{examp:2}
Let $\FF_p$ be $\FF_2$ and $L=[0;X,X^2,X,X^2,X,X^2,\ldots]$. Then $L$ solves 
$L^2+X^2L+X=0$. The digital sequence associated with $L$ is a $(1,1)$-sequence over $\FF_q$. 
The formal Laurent series of $L$ is of the form 
$$L=\sum_{n\geq 1}1/X^{2^{n+1}-2^{n-1}-2}.$$
\end{example}

For the proof of Theorem~\ref{thm:2} we show two basic properties: \\
- Let $$I:=[0,\gamma)\times \prod_{j=1}^t\left[\frac{c_j}{p^{e_jl_j}},\frac{c_j+1}{p^{e_jl_j}}\right)$$
with $0\leq c_j<e_jl_j$, $l_j\geq 0$ such that $\sum_{j=1}^te_jl_j\leq \log_p(N)/2$. Then
\begin{equation}\label{equ:1}|A_N(I)-N\lambda(I)|\ll_{L,p} \sqrt{N}\log N\end{equation}
- Let $$J:=[0,\gamma)\times \prod_{j=1}^t\left[\frac{c_j}{p^{e_jl_j}},\delta_j\right)\subseteq[0,\gamma)\times\prod_{j=1}^t\left[\frac{c_j}{p^{e_jl_j}},\frac{c_j+1}{p^{e_jl_j}}\right)\ $$
with $0\leq c_j<e_jl_j$, $l_j\geq 0$ such that $\sum_{j=1}^te_jl_j>\log_p(N)/2$. Then
\begin{equation}\label{equ:2}\max(A_N(J),N\lambda(J))\leq \sqrt{N}+1\end{equation}
For the second we see that $N\lambda(J)\leq N\frac{1}{p^{\sum_{j=1}^te_jl_j}}\leq  N\frac{1}{p^{\log_p(N)/2}}=N\frac{1}{\sqrt{N}}=\sqrt{N}$. Furthermore, $(x_n,\bsy_n)\in J$ implies $\bsy_n\in\prod_{j=1}^t\left[\frac{c_j}{p^{e_jl_j}},\frac{c_j+1}{p^{e_jl_j}}\right)$ which is equivalent to 
$$n(X)\equiv R(X)\pmod{\prod_{j=1}^tb_j^{l_j}(X)}$$ with a fixed $R(X)$ of degree $<\sum_{j=1}^te_jl_j$. Hence 
$$A_N(J)\leq \left\lceil\frac{N}{p^{\sum_{j=1}^te_jl_j}}\right\rceil\leq \left\lceil\sqrt{N}\right\rceil\leq \sqrt{N}+1.$$

For the first write $K(L)=\sup_{i\geq 1}\deg(A_i)$, $e:=\sum_{j=1}^te_jl_j$, and
\begin{align*}N=&N_{\lfloor \log_pN\rfloor}p^{\lfloor \log_pN\rfloor}+N_{\lfloor \log_pN\rfloor-1}p^{\lfloor \log_pN\rfloor-1}+\cdots+N_{K(L)+2e-1}p^{K(L)+2e-1}+\\
&+\underbrace{N_{K(L)+2e-2}p^{K(L)+2e-2}+\cdots+N_1p+N_0}_{=:M_1}
\end{align*} in base $p$. 
We have for the last $M_1$ points of the sequence $A_{M_1}(I)\leq \left\lceil\frac{p^{K(L)+2e-1}}{p^e}\right\rceil\ll p^{K(L)-1}\sqrt{N}$, and also $M_1\lambda(I)\ll p^{K(L)-1}\sqrt{N}$. 

The first $N_{\lfloor \log_pN\rfloor}p^{\lfloor \log_pN\rfloor}$ points relate to a subsequence of the Kronecker sequence $(x_{n_k})_{k\geq 0}$ that is determined by the indices $n$ such that $(x_n,\bsy_n)\in[0,1)\times \prod_{j=1}^t\left[\frac{c_j}{p^{e_jl_j}},\frac{c_j+1}{p^{e_jl_j}}\right)$. Hence by Lemma \ref{lem:2} and Proposition~\ref{prop:1} this $N_{\lfloor \log_pN\rfloor}p^{\lfloor \log_pN\rfloor}$ points of this subsequence 
form $N_{\lfloor \log_pN\rfloor}$ $$(K(L)+e-1,\lfloor\log_pN\rfloor-e,1)-\mbox{nets in base }p.$$
The next $N_{\lfloor \log_pN\rfloor-1}p^{\lfloor \log_pN\rfloor-1}$ points relate to the subsequence of the Kronecker sequence $(x_{n_k})_{k\geq 0}$ that is determined by the indices $n$ such that $(x_n,\bsy_n)\in[0,1)\times \prod_{j=1}^t\left[\frac{c_j}{p^{e_jl_j}},\frac{c_j+1}{p^{e_jl_j}}\right)$. Hence this $N_{\lfloor \log_pN\rfloor-1}p^{\lfloor \log_pN\rfloor-1}$ points belonging to this subsequence 
form $N_{\lfloor \log_pN\rfloor-1}$ $$(K(L)+e-1,\lfloor\log_pN\rfloor-e-1,1)-\mbox{nets in base }p.$$\\
...\\
And finally, the last $N_{K(L)+2e-1}p^{K(L)+2e-1}$ points relate to the subsequence of the Kronecker sequence $(x_{n_k})_{k\geq 0}$ that is determined by the indices $n$ such that $(x_n,\bsy_n)\in[0,1)\times \prod_{j=1}^t\left[\frac{c_j}{p^{e_jl_j}},\frac{c_j+1}{p^{e_jl_j}}\right)$. Hence this $N_{K(L)+2e-1}p^{K(L)+2e-1}$ points of this subsequence 
form $N_{K(L)+2e-1}$ $$(K(L)+e-1,K(L)+e-1,1)-\mbox{nets in base }p.$$
Each of these $\ll_p \log N$ nets satisfies $ND^*_N\ll p^{K(L)+e-1}$. Hence, 
\begin{align*}
|A_N(I)-N\lambda(I)|&\ll_p p^{K(L)+e-1}\log N\\
&\ll_L\sqrt{N}\log N.
\end{align*}
And the proof of the first item is complete. \\

Now for the proof of the Theorem~\ref{thm:2} we start with an arbitrary subinterval 
$$S=[0,\gamma]\times \prod_{j=1}^t[0,\delta_j).$$
We write $\delta_j$ in base $p^{e_j}$
$$\delta_j=\sum_{i=1}^\infty \frac{\beta_{j,i}}{p^{ie_j}}.$$
Define $n_j:=\max(l_j:l_je_j\leq\log_p(N))$, set $z_{j,0}=0$, $z_{j,l_j}:=\sum_{i=1}^{l_j} \frac{\beta_{j,i}}{p^{ie_j}}$ for $1\leq l_j\leq n_j+1$ and $z_{n_j+2}=\delta_j$. We split the above interval into the disjoint union
$$S= \bigcup_{l_1=1}^{n_1+2}\cdots \bigcup_{l_t=1}^{n_t+2}\underbrace{[0,\gamma)\times \prod_{j=1}^t[z_{j,l_j-1},z_{l_j})}_{=:I(l_1,\ldots,l_t)}.$$
Hence 
$$|A_N(S)-N\lambda(S)|\leq \sum_{l_1=1}^{n_1+2}\cdots \sum_{l_t=1}^{n_t+2}|A_{N}(I(l_1,\ldots,l_t))-N\lambda(I(l_1,\ldots,l_t))|=\Sigma_1+\Sigma_2$$
where $\Sigma_1$ sums over all $(l_1,\ldots,l_t)$ such that $\sum_{j=1}^te_jl_j\leq \log_p(N)/2$ and $\Sigma_2$ over the rest. Note that both sums have at most $\prod_{j=1}^t(n_j+2)\ll_{p,t,b_1,\ldots,b_t}\log^t N$ summands.\\

We regard a summand of $\Sigma_1$: 
Here 
\begin{align*}
I(l_1,\ldots,l_t)&=[0,\gamma)\times \prod_{j=1}^t\left[\sum_{i=1}^{l_j-1} \frac{\beta_{j,i}}{p^{ie_j}},\sum_{i=1}^{l_j} \frac{\beta_{j,i}}{p^{ie_j}}\right)\\
&=\bigcup_{c_1=0}^{\beta_{1,l_1}-1}\cdots \bigcup_{c_t=0}^{\beta_{1,l_t}-1}[0,\gamma)\times \prod_{j=1}^t\left[\sum_{i=1}^{l_j-1} \frac{\beta_{j,i}}{p^{ie_j}}+\frac{c_j}{p^{l_je_j}},\sum_{i=1}^{l_j-1} \frac{\beta_{j,i}}{p^{ie_j}}+\frac{c_j+1}{p^{l_je_j}}\right).
\end{align*}
The latter are intervals in the form of the first item. 
Hence 
\begin{align*}
|A_N(I(l_1,\ldots,l_t))-N\lambda(I(l_1,\ldots,l_t))|&\ll_{L,p} \sum_{c_1=0}^{\beta_{1,l_1}-1}\cdots \sum_{c_t=0}^{\beta_{1,l_t}-1}\sqrt{N}\log N\\
&\ll_{L,p,e_1,\ldots,e_s}\sqrt{N}\log N. 
\end{align*}
Altogether $\Sigma_1\ll_{L,p,t,b_1,\ldots,b_t}\sqrt{N}\log^{t+1} N$. \\

Finally, consider a summand of $\Sigma_2$: 
\begin{align*}
I(l_1,\ldots,l_t)&=[0,\gamma)\times \prod_{j=1}^t\left[z_{j,l_j-1},z_{j,l_j}\right)\\
&=\bigcup_{c_1=0}^{\beta_{1,l_1}}\cdots \bigcup_{c_t=0}^{\beta_{t,l_t}}\underbrace{[0,\gamma)\times \left[\sum_{i=1}^{l_j-1} \frac{\beta_{j,i}}{p^{ie_j}}+\frac{c_j}{p^{l_je_j}},\sum_{i=1}^{l_j-1} \frac{\beta_{j,i}}{p^{ie_j}}+\kappa_{j,l_j,c_j}\right)}_{=:J(c_1,\ldots,c_t)}
\end{align*}
with $\kappa_{j,l_j,c_j}=\delta_j-\sum_{i=1}^{l_j} \frac{\beta_{j,i}}{p^{ie_j}}$ if $l_j=n_j+2$ and $c_j=\beta_{j,l_j}$, $\kappa_{j,l_j,c_j}=\frac{c_j}{p^{l_je_j}}$ if $l_j<n_j+2$ and $c_j=\beta_{j,l_j}$, and $\kappa_{j,l_j,c_j}=\frac{c_j+1}{p^{l_je_j}}$ else. 
Thus by the second item 
\begin{align*}
|A_N(I(l_1,\ldots,l_t))-N\lambda(I(l_1,\ldots,l_t))|&\leq \sum_{c_1=0}^{\beta_{1,l_1}}\cdots \sum_{c_t=0}^{\beta_{t,l_t}}|A_N(J(c_1,\ldots,c_t))-N\lambda(J(c_1,\ldots,c_t))|\\
&\leq \sum_{c_1=0}^{\beta_{1,l_1}}\cdots \sum_{c_t=0}^{\beta_{t,l_t}}\max(A_N(J(c_1,\ldots,c_t)),N\lambda(J(c_1,\ldots,c_t)))\\
&\ll \sum_{c_1=0}^{\beta_{1,l_1}}\cdots \sum_{c_t=0}^{\beta_{t,l_t}}\sqrt{N}\\
&\ll_{p,e_1,\ldots,e_s}\sqrt{N}.
\end{align*}
Altogether 
$$\Sigma_2\ll_{p,t,b_1,\ldots,b_s}\sqrt{N} \log^t N$$
and the proof is complete.

\section{Proof of Theorem~\ref{thm:3}}\label{sec:5}

Let $n\in\NN$ and $N=2^{2^{n+2}-2^{n}-3}$. Then regard the elementary interval of the form 
$$I_n=[1/2,1)\times [0/2^{2^{n+1}-2^{n-1}-2},1/2^{2^{n+1}-2^{n-1}-2})\subseteq[0,1)^2.$$
Now $y_n\in[0/2^{2^{n+1}-2^{n-1}-2},1/2^{2^{n+1}-2^{n-1}-2})$ if and only if $X^{2^{n+1}-2^{n-1}-2}|n(X)$. So we regard the subsequence $(x_{n_k})_{0\leq k<2^{2^{n+1}-2^{n-1}-1}}$ that goes along with the polynomials $X^{2^{n+1}-2^{n-1}-2}k(X)$ where $\deg(k(X))<{2^{n+1}-2^{n-1}-1}$. 
Using the construction of $x_n$ and Example \ref{examp:2} we see 
\begin{align*}
(\underbrace{1,0,0,1,0,\ldots,0,1}_{\mbox{first ${2^{n+1}-2^{n-1}-2}$ coefficients of $L$,}}\underbrace{0,0,\ldots,0,}_{\mbox{next ${2^{n+1}-2^{n-1}-1}$ coefficients of $L$}} 1,0,\ldots)&\times \\
\times (\underbrace{0,\ldots,0}_{\mbox{first ${2^{n+1}-2^{n-1}-2}$ coefficients of $n(X)$,}},\underbrace{k_0,k_1,\ldots,k_{2^{n+1}-2^{n-1}-2}}_{\mbox{first ${2^{n+1}-2^{n-1}-1}$ coefficients of $k(X)$}},0,\ldots )^T=(0).& 
\end{align*}
Thus the interval $I_n$ remains empty. Hence
\begin{align*}
ND_N&\geq |A_N(I_n)-N\lambda(I_n)|=\frac{2^{2^{n+2}-2^{n}-3}}{2^{2^{n+1}-2^{n-1}-1}}\\
&=\gg\sqrt{N}.
\end{align*}

\section{Proof of Theorem~\ref{thm:4}}\label{sec:6}

For the proof of Theorem~\ref{thm:4} we collect several auxiliary results. 

Let $\mathcal{P}$ be the set of polynomials in $\FF_p[X]$ of degree $\geq1$. 
\begin{lemma}\label{lem:3}
Let $B_1,\ldots,B_k\in\mathcal{P}$ and let 
$$R(B_1,\ldots,B_k)=\{L\in\mathcal{H}:A_j(L)=B_j,\,1\leq j\leq k\}.$$
Then 
$h(R(B_1,\ldots,B_k))=p^{-2(\deg(B_1)+\cdots+\deg(B_k))}$. 
\end{lemma}
\begin{proof}
See \cite[Lemma~2]{Nie88b}. 
\end{proof}

Note that 
$$\sum_{B\in\mathcal{P}}p^{-2\deg(B)}=1.$$

\begin{lemma}\label{lem:4}
Let $B\in\FF_p[X]\setminus\{0\}$. Then $f:\mathcal{H}\to\mathcal{H},\,L\mapsto \{BL\}$ is $h$-measure preserving. 
\end{lemma}
\begin{proof}
See, e.g., \cite{Spr} or \cite{Han}. 
\end{proof}

\begin{lemma}\label{lem:5}
Let $L(X)=\sum_{k=w}^\infty a_kX^{-k}$ be any formal Laurent series over $\FF_p$. Let $m\in\NN$, $d_h:=\deg(Q_h)$, and $H\in\NN$ such that $d_H\leq m<d_{H+1}$. Then 
the matrix 
$$\begin{pmatrix}
a_1&a_2&\cdots&a_m\\
a_2&a_3&\cdots&a_{m+1}\\
\ddots&\ddots&\ddots&\ddots\\
a_{d_H}&a_{d_H+1}&\ddots&a_{m+d_H-1} 
\end{pmatrix} $$
over $\FF_p$ has full row rank. 
\end{lemma}
\begin{proof}
We know $\nu(\{Q_{H-1}L\})=-d_{H}$ and for all $P\in\FF_p[X]\setminus\{0\}$ with $\deg(P)<d_{H}$ we have $\nu(\{PL\})\geq \nu(\{Q_{H-1}L\})=-d_{H}$. Suppose that the rows are linearly dependent, then there exists a $P(X)=\sum_{r=0}^{d_H-1}p_rX^r$ not the zero polynomial such that $\nu(\{PL\})<-m\leq -d_H$, which is a contradiction. 
\end{proof}

\begin{lemma}\label{lem:6}
Let $V(X),L(X)$ be any formal Laurent series over $\FF_p$. Let $m\in\NN$, $d_h:=\deg(Q_h)$, and $H\in\NN$ such that $d_H\leq m<d_{H+1}$. Then  
$$p^mD^*_{p^m}(\{k(x)L(x)+V(x)\})\leq p^{\deg(A_{H+1}(L))}.$$
\end{lemma}
\begin{proof}
By the last lemma this point set is a $(m-d_H,m,1)$-net. Hence 
$$p^mD^*_{p^m}\leq p^{m-d_H}\leq p^{d_{H+1}-d_H}=p^{\deg(A_{H+1}(L))}.$$
\end{proof}

One of the core results in the proof of Theorem~\ref{thm:4} is the following proposition. 
\begin{prop}\label{prop:2}
Let $(x_n)_{n\geq 0}$ be the Kronecker sequence determined by $L\in\mathcal{H}$, let $(\bsy_n)_{n\geq 0}$ be the Halton sequence in bases $b_1(X),\ldots,b_t(X)$ pairwise coprime nonconstant and monic. Then for all $N>1$
$$ND_N((x_n,\bsy_n))\ll_{b_1,\ldots,b_t,p,t} \log^t N +\sum_{h=1}^{\lfloor\log_pN\rfloor}\sum_{l_1=1}^{\lfloor\log_pN\rfloor}\cdots \sum_{l_t=1}^{\lfloor\log_pN\rfloor}\deg(A_h(b_1^{l_1}\cdots b_t^{l_t}L))p^{\deg(A_h(b_1^{l_1}\cdots b_t^{l_t}L))}$$
\end{prop}
\begin{proof}
We start with an arbitrary subinterval $S:=[0,\alpha)\times \prod_{j=1}^t[0,\beta_j)\subseteq[0,1)^{t+1}$. 
 
We write $\beta_j$ in base $p^{e_j}$: $\beta_j =\sum_{i=1}^\infty \frac{\beta_{j,i}}{p^{e_ji}}$. Define $\eta_j:=\lfloor\log_{p^{e_j}}N\rfloor$, $z_{j,0}=0$, $z_{j,l}=\sum_{i=1}^l \frac{\beta_{j,i}}{p^{e_ji}}$ for $l=1,\ldots,\eta_j+1$, and $z_{j,\eta_j+2}=\beta_j$. Then 
$$S=\bigcup_{l_1=1}^{\eta_1+2}\cdots \bigcup_{l_t=1}^{\eta_t+2}\underbrace{[0,\alpha)\times \prod_{j=1}^t[z_{j,l_j-1},z_{j,l_j}]}_{I(l_1,\ldots,l_s)}$$
and 
$$|A_N(S)-N\lambda(S)|\leq \sum_{l_1=1}^{\eta_1+2}\cdots \sum_{l_t=1}^{\eta_t+2}|A_N(I(l_1,\ldots,l_s))-N\lambda(I(l_1,\ldots,l_s))|=:\Sigma_1+\Sigma_2$$
where $\Sigma_1$ sums over all $(l_1,\ldots,l_t)$ such that $\sum_{j=1}^te_jl_j\leq \log_p N$ and $\Sigma_2$ over the rest. Note that both have at most $\prod_{j=1}^t(\eta_j+2)\ll_{p,t,b_1,\ldots,b_t} \log^t N $ summands. 

Let us first consider $\Sigma_2$: We split
\begin{align*}
I(l_1,\ldots,l_t)&=[0,\alpha)\times \prod_{j=1}^t\left[z_{j,l_j-1},z_{j,l_j}\right)\\
&=\bigcup_{c_1=0}^{\beta_{1,l_1}}\cdots \bigcup_{c_t=0}^{\beta_{t,l_t}}\underbrace{[0,\alpha)\times \left[\sum_{i=1}^{l_j-1} \frac{\beta_{j,i}}{p^{ie_j}}+\frac{c_j}{p^{l_je_j}},\sum_{i=1}^{l_j-1} \frac{\beta_{j,i}}{p^{ie_j}}+\kappa_{j,l_j,c_j}\right)}_{=:J(c_1,\ldots,c_t)}
\end{align*}
with $\kappa_{j,l_j,c_j}=\beta_j-\sum_{i=1}^{l_j} \frac{\beta_{j,i}}{p^{ie_j}}$ if $l_j=n_j+2$ and $c_j=\beta_{j,l_j}$, $\kappa_{j,l_j,c_j}=\frac{c_j}{p^{l_je_j}}$ if $l_j<n_j+2$ and $c_j=\beta_{j,l_j}$, and $\kappa_{j,l_j,c_j}=\frac{c_j+1}{p^{l_je_j}}$ else. 
Thus  
\begin{align*}
|A_N(I(l_1,\ldots,l_t))-N\lambda(I(l_1,\ldots,l_t))|&\leq \sum_{c_1=0}^{\beta_{1,l_1}}\cdots \sum_{c_t=0}^{\beta_{t,l_t}}|A_N(J(c_1,\ldots,c_t))-N\lambda(J(c_1,\ldots,c_t))|\\
&\leq \sum_{c_1=0}^{\beta_{1,l_1}}\cdots \sum_{c_t=0}^{\beta_{t,l_t}}\max(A_N(J(c_1,\ldots,c_t)),N\lambda(J(c_1,\ldots,c_t)))\\
&\leq \sum_{c_1=0}^{\beta_{1,l_1}}\cdots \sum_{c_t=0}^{\beta_{t,l_t}}1\\
&\ll_{p,e_1,\ldots,e_s}1
\end{align*}
where we used that $\lambda(J(c_1,\ldots,c_t))\leq 1/N$ and $A_N(J(c_1,\ldots,c_t))\leq 1$. 

Altogether 
$$\Sigma_2\ll_{p,t,b_1,\ldots,b_s} \log^t N.$$
It remains to estimate $\Sigma_1$: 
Here 
\begin{align*}
I(l_1,\ldots,l_t)&=[0,\alpha)\times \prod_{j=1}^t\left[\sum_{i=1}^{l_j-1} \frac{\beta_{j,i}}{p^{ie_j}},\sum_{i=1}^{l_j} \frac{\beta_{j,i}}{p^{ie_j}}\right)\\
&=\bigcup_{c_1=0}^{\beta_{1,l_1}-1}\cdots \bigcup_{c_t=0}^{\beta_{1,l_t}-1}\underbrace{[0,\alpha)\times \prod_{j=1}^t\left[\sum_{i=1}^{l_j-1} \frac{\beta_{j,i}}{p^{ie_j}}+\frac{c_j}{p^{l_je_j}},\sum_{i=1}^{l_j-1} \frac{\beta_{j,i}}{p^{ie_j}}+\frac{c_j+1}{p^{l_je_j}}\right)}_{=:J(c_1,\ldots,c_t)}.
\end{align*}
Hence 
\begin{align*}
|A_N(I(l_1,\ldots,l_t))-N\lambda(I(l_1,\ldots,l_t))|&\leq \sum_{c_1=0}^{\beta_{1,l_1}-1}\cdots \sum_{c_t=0}^{\beta_{1,l_t}-1}|A_N(J(c_1,\ldots,c_t))-N\lambda(J(c_1,\ldots,c_t))|.
\end{align*}
Now we have $\bsy_n\in\prod_{j=1}^t\left[\sum_{i=1}^{l_j-1} \frac{\beta_{j,i}}{p^{ie_j}}+\frac{c_j}{p^{l_je_j}},\sum_{i=1}^{l_j-1} \frac{\beta_{j,i}}{p^{ie_j}}+\frac{c_j+1}{p^{l_je_j}}\right)$ if and only if 
\begin{equation}\label{equ:star}n(X)\equiv R(X)\pmod{\underbrace{\prod_{j=1}^tb_j(X)^{l_j}}_{B(X)}}\end{equation}
where $R(X)$ is determined by the $\beta_{i,j}$ and $c_j$. We set $e:=\sum_{j=1}^t{e_jl_j}$ and write $$N=N_0+N_1p+N_2p^2+\cdots+N_{e-1}p^{e-1}+N_{e}p^{e}+N_{e+1}p^{e+1}+\cdots+N_{\lfloor\log_p N\rfloor}p^{\lfloor\log_p N\rfloor}.$$

We regard the indizes $n=0,1,\ldots, p^{\lfloor\log_p N\rfloor}-1$. Then for the $p^{\lfloor\log_p N\rfloor-e}$ points $x_n$ with $n$ satisfying \eqref{equ:star}, by Lemma \ref{lem:2} and \ref{lem:6}, we obtain
\begin{align*}
|A_{p^{\lfloor\log_p N\rfloor-e}}([0,\alpha))-p^{\lfloor\log_p N\rfloor-e}\lambda([0,\alpha))|&\leq p^{\lfloor\log_p N\rfloor-e}D^*_{p^{\lfloor\log_p N\rfloor-e}}(\{k(X)B(X)L(X)+V(X)\})\\&\leq p^{\deg(A_{H+1}(BL))}
\end{align*}
for some $V(X)$, 
where $H$ is such that $d_H\leq \lfloor\log_p N\rfloor-e<d_{H+1}$. 
We proceed step by step and end up if $N_{e}\geq 1$ with  $n=(N_{e}-1)p^{e}+N_{e+1}p^{e+1}+\cdots+N_{\lfloor\log_p N\rfloor}p^{\lfloor\log_p N\rfloor}, (N_{e}-1)p^{e}+N_{e+1}p^{e+1}+\cdots+N_{\lfloor\log_p N\rfloor}p^{\lfloor\log_p N\rfloor}+1,\ldots, (N_{e}-1)p^{e}+N_{e+1}p^{e+1}+\cdots+N_{\lfloor\log_p N\rfloor}p^{\lfloor\log_p N\rfloor}+p^{e}-1$. Then for the one point $x_n$ with $n$ satisfying \eqref{equ:star} we have
$$|A_1([0,\alpha)-1\lambda([0,\alpha))|\leq 1 \leq p^{\deg(A_{1}(BL))}.$$
Note the fact that $d_{h+1}-d_h=\deg(A_{h+1}(BL))$, then trivially $H\leq \log_p N$, and note also the fact that $l_j\leq \log_p N$. Thus, it is not so hard to see that the sum in the proposition together with the implied constant is a proper upper bound. 

For the last $N_0+N_1p+N_2p^2+\cdots+N_{e-1}p^{e-1}$ points we obtain 
$$|A_{N_0+N_1p+N_2p^2+\cdots+N_{e-1}p^{e-1}}(J(c_1,\ldots,c_t))-(N_0+N_1p+N_2p^2+\cdots+N_{e-1}p^{e-1})\lambda(J(c_1,\ldots,c_t))|\leq 1. $$
Those terms end up in $\ll \log^t N$. 

\end{proof}

The second core result for the proof of Theorem~\ref{thm:4} is the following. 

\begin{prop}\label{prop:3} We have
$$\sum_{h=1}^{\lfloor\log_pN\rfloor}\sum_{l_1=1}^{\lfloor\log_pN\rfloor}\cdots \sum_{l_t=1}^{\lfloor\log_pN\rfloor}\deg(A_h(b_1^{l_1}\cdots b_t^{l_t}L))p^{\deg(A_h(b_1^{l_1}\cdots b_t^{l_t}L))}\ll_{t,L,\epsilon,p}\log^{t+1+\epsilon} N$$
for all $N>1$ and all $\epsilon>0$ for almost all $L\in\mathcal{H}$ in the sense of Haar-measure. 
\end{prop}
\begin{proof}
Let 
$$E_C:=\{L\in\mathcal{H}:p^{\deg(A_h(b_1^{l_1}\cdots b_t^{l_t}L))}< C(\overline{l}_1\cdots \overline{l}_th)^4 \mbox{ for all $l_1,\ldots,l_t\in\NN_0$ and all $h\in\NN$}\},$$
where $\overline{l}=\max(1,l)$. 

We first show $\lim_{C\to\infty}h(E_C)=1$. 

The complement $\overline{E}_C$ of $E_C$ contains all $L$ for which there exists $h_0$ and $l_{1,0},\ldots,l_{t,0}$ such that $$p^{\deg(A_{h_0}(b_1^{l_{1,0}}\cdots b_t^{l_{t,0}}L))}\geq C(\overline{l}_{1,0}\cdots \overline{l}_{t,0}h_{0})^4.$$
For such $L$ denote by $h_0(L)$ the minimal such $h_0$, by $l_{1,0}(L)$ the minimal $l_{1,0}$ for given $h_0(L)$, further by $l_{2,0}(L)$ the minimal $l_{2,0}$ for given $h_0(L)$ and $l_{1,0}(L)$ and so on. Then 
$$\overline{E}_C=\bigcup_{h_0=1}^\infty\bigcup_{l_{1,0}=0}^\infty \cdots\bigcup_{l_{t,0}=0}^\infty \left\{L\in\overline{E}_C:h_0(L)=h_0,\,l_{1,0}(L)=l_{1,0},\,\ldots, l_{t,0}(L)=l_{t,0} \right\}.$$
We consider the Haar-measure of $$S= \left\{L\in\overline{E}_C:h_0(L)=h_0,\,l_{1,0}(L)=l_{1,0},\,\ldots, l_{t,0}(L)=l_{t,0} \right\}.$$
Let $L\in S$ then the $h_0-1$st convergent of $b_1^{l_{1,0}}\cdots b_t^{l_{t,0}}L$ satisfies 
$$\nu(b_1^{l_{1,0}}\cdots b_t^{l_{t,0}}L-\frac{P_{h_0-1}}{Q_{h_0-1}})=-2\deg(Q_{h_0-1})-\deg(A_{h_0}(b_1^{l_{1,0}}\cdots b_t^{l_{t,0}}L)).$$
Then 
$$\nu(L-\frac{A}{b_1^{l_{1,0}}\cdots b_t^{l_{t,0}}Q_{h_0-1}})=-2\deg(Q_{h_0-1})-\deg(A_{h_0}(b_1^{l_{1,0}}\cdots b_t^{l_{t,0}}L))-\deg(b_1^{l_{1,0}}\cdots b_t^{l_{t,0}})$$
for some $A$ with $\deg(A)<\deg(Q_{h_0-1})+\deg(b_1^{l_{1,0}}\cdots b_t^{l_{t,0}})$. 
This means $L$ is contained in a set of Haar measure 
\begin{align*}
p^{-\deg(Q_{h_0-1})-A_{h_0}(b_1^{l_{1,0}}\cdots b_t^{l_{t,0}}L)}&\leq  p^{-\deg(Q_{h_0-1})}\frac{1}{C(\overline{l}_{1,0}\cdots \overline{l}_{t,0}h_0)^4}
\end{align*}
where we used $\deg(A_{h_0}(b_1^{l_{1,0}}\cdots b_t^{l_{t,0}}L))\geq \log_{p}(C(\overline{l}_{1,0}\cdots \overline{l}_{t,0}h_0)^4)$. 
Note that we have $\deg(Q_{h_0-1})=\sum_{i=1}^{h_0-1}\deg(A_i(b_1^{l_{1,0}}\cdots b_t^{l_{t,0}}L))\leq h_0\log_{p}(C(\overline{l}_{1,0}\cdots \overline{l}_{t,0}h_0)^4)$. Hence 
\begin{align*}
h(S)&<\sum_{Q\in\mathcal{P},\deg(Q)\leq h_0\log_{p}(C(\overline{l}_{1,0}\cdots \overline{l}_{t,0}h_0)^4)^4)}p^{-\deg(Q)}\frac{1}{C(\overline{l}_{1,0}\cdots \overline{l}_{t,0}h_0)^4}\\
&<(p-1)\frac{h_0\log_{p}(C(\overline{l}_{1,0}\cdots \overline{l}_{t,0}h_0)^4)}{C(\overline{l}_{1,0}\cdots \overline{l}_{t,0}h_0)^4}.
\end{align*}
Hence 
\begin{align*}
h(\overline{E}_C)&\leq \sum_{h=1}^\infty \sum_{l_1=0}^\infty\dots\sum_{l_t=0}^\infty (p-1)\frac{h\log_{p}(C(\overline{l}_1\cdots \overline{l}_th)^4)}{C(\overline{l}_1\cdots \overline{l}_th)^4}\ll_p\log C/C  
\end{align*}
and the assertion $\lim_{C\to\infty}h(E_C)=1$ holds. \\

Next we show $\int_{E_C}\deg(A_k(b_1^{l_{1}}\cdots b_t^{l_{t}}L))p^{\deg(A_k(b_1^{l_{1}}\cdots b_t^{l_{t}}L))}dh(L)\ll_{p,C} \log^2_p((\overline{l}_1\cdots\overline{l}_th))$: 

We define elementary intervals 
$$I_{B_1,\ldots,B_k}^{(l_1,\ldots,l_s)}:=\{L\in\mathcal{H}:A_h(b_1^{l_{1}}\cdots b_t^{l_{t}}L)=B_h\mbox{ for all $h=1,\ldots,k$}\}.$$
By Lemma \ref{lem:3} and \ref{lem:4} we have 
$$h(I_{B_1,\ldots,B_k}^{(l_1,\ldots,l_s)})=h(I_{B_1,\ldots,B_k}^{(0,\ldots,0)})=p^{-2(\deg(B_1)+\cdots+\deg(B_k))}$$
and the identity $$h(I_{B_1,\ldots,B_k}^{(l_1,\ldots,l_s)})=p^{-2(\deg(B_k))}h(I_{B_1,\ldots,B_{k-1}}^{(l_1,\ldots,l_s)}).$$
Hence by definition $\int_{E_C}\deg(A_k(b_1^{l_{1}}\cdots b_t^{l_{t}}L))p^{\deg(A_k(b_1^{l_{1}}\cdots b_t^{l_{t}}L))}dh(L)$
\begin{align*}
&\leq \sum_{B_k\in\mathcal{P},\deg(B_k)\leq \log_p(C(\overline{l}_1\cdots\overline{l}_tk)^4)}\sum_{B_1\in\mathcal{P}}\cdots\sum_{B_{k-1}\in\mathcal{P}}\int_{I_{B_1,\ldots,B_k}^{(l_1,\ldots,l_s)}}\deg(A_k(b_1^{l_{1}}\cdots b_t^{l_{t}}L))p^{\deg(A_k(b_1^{l_{1}}\cdots b_t^{l_{t}}L))}dh(L)\\
&=\sum_{B_k\in\mathcal{P},\deg(B_k)\leq \log_p(C(\overline{l}_1\cdots\overline{l}_tk)^4)}\sum_{B_1\in\mathcal{P}}\cdots\sum_{B_{k-1}\in\mathcal{P}}\deg(B_k)p^{\deg(B_k)}h(I_{B_1,\ldots,B_k}^{(l_1,\ldots,l_s)})\\
&=\sum_{B_k\in\mathcal{P},\deg(B_k)\leq \log_p(C(\overline{l}_1\cdots\overline{l}_tk)^4)}\deg(B_k)p^{-(\deg(B_k))}\sum_{B_1\in\mathcal{P}}\cdots\sum_{B_{k-1}\in\mathcal{P}}h(I_{B_1,\ldots,B_{k-1}}^{(l_1,\ldots,l_s)})\\
&=\sum_{B_k\in\mathcal{P},\deg(B_k)\leq \log_p(C(\overline{l}_1\cdots\overline{l}_tk)^4)}\deg(B_k)p^{-(\deg(B_k))}\\
&\leq\log_p(C(\overline{l}_1\cdots\overline{l}_tk)^4)\sum_{r=1}^{\log_p(C(\overline{l}_1\cdots\overline{l}_tk)^4)}(p-1)p^rp^{-r}\\
&\ll_{t,p,C} \log^2_p(\overline{l}_1\cdots\overline{l}_tk).
\end{align*}

Now we estimate 
\begin{align*}
\int_{E_C}\sum_{h=1}^{H}\sum_{l_1=1}^{H}\cdots \sum_{l_t=1}^{H}\frac{\deg(A_h(b_1^{l_1}\cdots b_t^{l_t}L))p^{\deg(A_h(b_1^{l_1}\cdots b_t^{l_t}L))}}{(\overline{l}_1\cdots \overline{l}_s  h)^{1+\epsilon/(t+1)}}dh(L)&\ll_{t,p,C} \sum_{h=1}^{H}\sum_{l_1=1}^{H}\cdots \sum_{l_t=1}^{H}\frac{\log^2_p(\overline{l}_1\cdots\overline{l}_th)}{(\overline{l}_1\cdots \overline{l}_s  h)^{1+\epsilon/(t+1)}}\\&\ll_{t,p,C,\epsilon}1.
\end{align*}
Since the integrand above is monotonically increasing in $H$, we have for almost all $L\in\mathcal{H}$ in the sense of Haar measure that 
$$\sum_{h=1}^{H}\sum_{l_1=1}^{H}\cdots \sum_{l_t=1}^{H}\frac{\deg(A_h(b_1^{l_1}\cdots b_t^{l_t}L))p^{\deg(A_h(b_1^{l_1}\cdots b_t^{l_t}L))}}{(\overline{l}_1\cdots \overline{l}_s  h)^{1+\epsilon/(t+1)}}\ll_{t,p,C,\epsilon}1,$$
and hence for almost all $L\in\mathcal{H}$
$$\sum_{h=1}^{H}\sum_{l_1=1}^{H}\cdots \sum_{l_t=1}^{H}{\deg(A_h(b_1^{l_1}\cdots b_t^{l_t}L))p^{\deg(A_h(b_1^{l_1}\cdots b_t^{l_t}L))}}\ll_{t,p,C,\epsilon}H^{t+1+\epsilon}$$
for all $H$. 
\end{proof}

Theorem~\ref{thm:4} is now an immediate consequence of Proposition~\ref{prop:2}
and \ref{prop:3}.

\noindent{Roswitha Hofer, Institute of Financial Mathematics and Applied Number Theory, Johannes Kepler University Linz, Altenbergerstr. 69, 4040 Linz, AUSTRIA, roswitha.hofer@jku.at}

\end{document}